\documentclass[a4paper,12pt]{article}
\usepackage{amsmath,amssymb,amsthm,verbatim,amsxtra}
\newcommand{\C}{{\mathbb C}}
\newcommand{\N}{{\mathbb N}}
\newcommand{\R}{{\mathbb R}}

\newcommand{\abs}[2][\empty]{\ifx#1\empty\left|#2\right|%
\else#1\vert #2 #1\vert\fi}
\newcommand{\Cnt}[1][]{{\cal C}^{#1}}
\newcommand{\csub}{\subset\subset}
\newcommand{\defstyle}[1]{{\bf #1}}
\renewcommand{\emptyset}{\varnothing}
\newcommand{\eps}{\varepsilon}
\renewcommand{\implies}{\Rightarrow}

\newcommand{\norm}[2][\empty]{\ifx#1\empty\left\Vert#2\right\Vert%
\else#1\Vert #2 #1\Vert\fi}

\newcommand{\Filter}{\mathcal F}
\newcommand{\Powerset}{\mathcal P}
\renewcommand{\phi}{\varphi}
\newcommand{\Gen}{{\mathcal G}}

\newcommand{\GenR}{\widetilde\R}

\newcommand{\Mod}{{\mathcal M}}
\newcommand{\Null}{{\mathcal N}}
\newcommand{\caninf}{\rho}

\newcommand{\ster}[1]{{{}^* \mskip-1mu #1}}
\newcommand{\stermathop}{\sideset{^*}{}}
\newcommand{\Fin}{\mathop{\mathrm{Fin}}}

\newtheorem*{df}{Definition}
\newtheorem{thm}{Theorem}[section]
\newtheorem{lemma}[thm]{Lemma}
\newtheorem{prop}[thm]{Proposition}
\newtheorem{cor}[thm]{Corollary}
\newtheorem{ex}[thm]{Example}

\theoremstyle{definition}
\newtheorem*{rem}{Remark}
\newtheorem{exc}[thm]{Exercise}
\newtheorem*{nota}{Notation}

\begin{document}
\title{Nonstandard principles for generalized functions}
\author{Lecture notes\footnote{As presented at the Summer school \emph{Generalized Functions in PDE, Geometry, Stochastics and Microlocal Analysis}, Novi Sad, August--September 2010}\\H.\ Vernaeve}
\date{}
\maketitle

Someone knowledgeable in nonstandard analysis may get the feeling that in the nonlinear theory of generalized functions, too often one works directly on the nets and spends effort to obtain results that should be clear from general principles. We want to show that such principles can indeed be introduced and to illustrate their role to solve problems.\\
This text is intended as a tutorial on the use of nonstandard principles in generalized function theory intended for researchers in the nonlinear theory of generalized functions.

\section{Generalized and internal objects}\label{section_internal_objects}
We will define generalized objects by means of families (nets) of objects indexed by $(0,1)$, identified if they coincide for small $\eps$. The objects thus defined are simpler than the corresponding Colombeau objects \cite{Col84,GKOS,Ob92} and can be viewed as `raw material' from which Colombeau objects can be constructed (or also as a `test environment' for making conjectures). E.g., the generalized numbers that will be defined immediately, are the generalized constants of Egorov's algebra of generalized functions.

\begin{df}
A \defstyle{generalized real number} is an equivalence class of nets $(a_\eps)_\eps$ up to equality for small $\eps$. Formally, the set of generalized real numbers equals
\[\ster\R := \R^{(0,1)}/\{(a_\eps)_\eps\in \R^{(0,1)}: a_\eps = 0 \text{ for small }\eps\}.\]
We denote the equivalence class of the net $(a_\eps)_\eps$ by $[a_\eps]$.
\end{df}

By means of elementary set-theoretic operations, we now define generalized objects by means of nets, the so-called \defstyle{internal objects}:
\begin{enumerate}
\item[0.] By definition, generalized real numbers are internal.
\item Let $m\in\N$, $m>1$. To a net $((a_{1,\eps}, \dots, a_{m,\eps}))_\eps$ of $m$-tuples of real numbers, we associate the $m$-tuple of generalized real numbers
\[[(a_{1,\eps}, \dots, a_{m,\eps})] := ([a_{1,\eps}], \dots, [a_{m,\eps}]).\]
\item To a net $(f_\eps)_\eps$ of maps $\R\to \R$, we associate the map $\ster\R\to\ster\R$
\[[f_\eps]([x_\eps]) := [f_\eps(x_\eps)],\quad \forall [x_\eps]\in \ster\R.\]
\item To a net $(R_\eps)_\eps$ of binary relations defined on $\R\times\R$, we associate the binary relation on $\ster\R\times \ster\R$
\[[x_{\eps}] [R_\eps] [y_{\eps}]\iff x_\eps R_\eps y_\eps \text{ for small }\eps.\]
\item To a net $(A_\eps)_\eps$ of \emph{nonempty} sets of real numbers, we associate the set of generalized real numbers
\[[A_\eps]:=\{[a_\eps] : a_\eps\in A_\eps \text{ for small  }\eps\}.\]
\end{enumerate}
Any $m$-tuple of generalized numbers is internal (i.e., associated to a net of $m$-tuples of real numbers). On the other hand, as we will see, not every set of generalized numbers is internal, and neither is every map $\ster\R\to\ster\R$.

\begin{df}
For $\emptyset\ne A\subseteq\R$, we denote $\ster A :=[A]$ (= the internal object corresponding to the constant net $(A)_\eps$).
\end{df}
\begin{exc}\leavevmode
\begin{enumerate}
\item The new definition of $\ster\R$ coincides with the original one.
\item $[(a_{1,\eps}, \dots, a_{m,\eps})]=[(b_{1,\eps}, \dots, b_{m,\eps})]$ iff the nets coincide for small $\eps$.
\item Let $\varnothing\ne A_\eps\subseteq \R$ and $\varnothing\ne B_\eps\subseteq\R$. Then
\[[A_\eps]\subseteq [B_\eps] \iff A_\eps \subseteq B_\eps \text{ for small }\eps\]
(use the fact that $A_\eps\ne\varnothing$!).\\
Conclude that $[A_\eps]= [B_\eps]$ iff $A_\eps = B_\eps$ for small $\eps$.
\end{enumerate}
\end{exc}

We can inductively repeat these rules of construction to define more internal objects:
\begin{enumerate}
\item Let $m\in\N$, $m>1$. If $[a_{1,\eps}]$, \dots, $[a_{m,\eps}]$ are internal objects, then
\[[(a_{1,\eps}, \dots, a_{m,\eps})] := ([a_{1,\eps}], \dots, [a_{m,\eps}]).\]
\item Let $[A_\eps]$, $[B_\eps]$ be internal sets. If $f_\eps$ are maps $A_\eps\to B_\eps$, then
\[[f_\eps]: [A_\eps]\to [B_\eps]: [f_\eps]([x_\eps]) := [f_\eps(x_\eps)].\]
\item Let $[A_\eps]$, $[B_\eps]$ be internal sets. If $R_\eps$ are binary relations defined on $A_\eps\times B_\eps$, then $[R_\eps]$ is the binary relation on $[A_\eps]\times [B_\eps]$ defined by
\[[x_{\eps}] [R_\eps] [y_{\eps}]\iff x_\eps R_\eps y_\eps \text{ for small }\eps.\]
\item If $A_\eps$ are nonempty sets such that for each $a_\eps\in A_\eps$, $[a_\eps]$ has already been defined, then
\[[A_\eps]:=\{[a_\eps] : a_\eps\in A_\eps \text{ for small  }\eps\}.\]
\end{enumerate}

\begin{df}
The union of the objects thus defined 
is the class of \defstyle{internal objects}. In particular, by convention, the empty set is not internal (=external). We also extend the definition $\ster a:=[a]$ for the new types of objects.
\end{df}
Also for $a\in\R^m$ ($m\ge 1$), $\ster a= [a]\in\ster\R^m$. In practice, we will identify $\R^m$ as a subset of $\ster\R^m$, and therefore drop the stars in this case.

\begin{itemize}
\item[E.g.,]
\item If $a_\eps\in \R$, $\varnothing\ne A_\eps\subseteq\R$, and $f_\eps$ are maps $\R\to\R$, then
\[[(a_{\eps}, A_\eps, f_{\eps})] = ([a_{\eps}], [A_\eps], [f_{\eps}]).\]
\item If $A_\eps$ are nonempty subsets of $\R^m$, then
\[[A_\eps] = \{[a_\eps] : a_\eps\in A_\eps \text{ for small  }\eps\}\subseteq \ster\R^m.\]
\item If $f_\eps$ are maps $\R^m\to\R$, then
\[[f_\eps]: \ster\R^m \to\ster\R: [f_\eps]([x_\eps]) = [f_\eps(x_\eps)].\]
\item If $A_\eps$ are nonempty sets of maps $\R^m\to\R$, then
\[
[A_\eps] = \{[f_\eps]: f_\eps\in A_\eps \text{ for small }\eps\}.
\]
\item If $T_\eps$: $\Cnt[\infty](\R^m)\to\Cnt[\infty](\R^m)$, then
\[[T_\eps]: \ster{\Cnt[\infty]}(\R^m)\to\ster{\Cnt[\infty]}(\R^m): [T_\eps]([f_\eps]) := [T_\eps(f_\eps)].\]
\item \dots
\end{itemize}

\begin{exc}
\begin{enumerate}\leavevmode
\item Denote the graph of a map $f$: $\R\to\R$ by $G_f:= \{(x, f(x)): x\in \R\}$. Let $f_\eps$, $g_\eps$ be maps $\R\to\R$. Then $G_{[f_\eps]} = [G_{f_\eps}]$. Conclude that
\[[f_\eps] = [g_\eps] \iff f_\eps = g_\eps\text{ for small } \eps.\]
\item Denote the graph of a relation $R$ on $\R\times\R$ by $G_R:= \{(x, y): x R\, y\}$. Let $R_\eps$, $S_\eps$ be relations on $\R\times\R$. Then $G_{[R_\eps]} = [G_{R_\eps}]$. Conclude that
\[[R_\eps] = [S_\eps] \iff R_\eps = S_\eps\text{ for small } \eps.\]
\end{enumerate}
\end{exc}

More generally, we have in all cases:
\begin{align*}
[a_\eps]\in [A_\eps] &\iff a_\eps\in A_\eps \text{ for small }\eps\\
[a_\eps] = [b_\eps] &\iff a_\eps = b_\eps \text{ for small }\eps\\
[a_\eps][R_\eps][b_\eps] &\iff a_\eps R_\eps b_\eps \text{ for small }\eps.
\end{align*}
This tells us that, although we are basically just working with nets up to equality for small $\eps$, they can be identified with a large class of generalized objects in a generic way.

\begin{exc}
We denote by $\Powerset(A)$ the set of all \emph{nonempty} subsets of $A$.\\
(a) An internal set contains only internal objects.\footnote{To be precise, we consider (generalized) numbers to be elementless (as usual  in analysis).}\\
(b) Let $A$ be a nongeneralized nonempty set. Then $\ster\Powerset(A)$ is the set of all internal subsets of $\ster A$. (In particular, $\ster\Powerset(\R)\subsetneqq\Powerset(\ster\R)$).\\
(c) An internal object defined by a net of elements of $A$ is contained in $\ster A$.\\
(d) Let $A$, $B$ be nongeneralized nonempty sets. Then $\ster(A\times B)=\ster A \times \ster B$. In particular, $\ster (A^m)=(\ster A)^m$ (which we will therefore denote by $\ster A^m$).\\
(e) Let $B^A$ be the set of all maps $A\to B$. Then $\ster{(B^A)}$ is the set of all internal maps $\ster A\to \ster B$.
\end{exc}

\section{Some examples in $\ster\R$}
A lot of the structure of $\R$ can be transferred to $\ster\R$. E.g., $+$: $\ster\R\times\ster\R\to\ster\R$ can be defined $\eps$-wise, i.e., coinciding with the map $\ster +$ from the generic construction. In the case of maps on $\ster\R$, we will usually drop the stars, since they extend the usual operations on $\R$ (identifying $\R$ with a subset of $\ster\R$). On representatives, one easily sees that $+$, $\cdot$ are associative and commutative. In fact, for any binary operation $f$ on $\R$, the statement
\[(\forall x,y\in \R)(f(x,y) = f(y,x))\]
transfers to
\[(\forall x,y\in \ster\R)(f(x,y) = f(y,x)).\]
Similarly, one sees that $\ster\R$ is an ordered commutative ring.\\
In the case of (binary, say) relations on $\R$, some confusion may arise in dropping the stars. E.g., for $a,b\in\ster\R$, $a(\ster\!\!\ne) b$ is not equivalent with $\neg(a=b)$. We will drop the stars for $\le$; on the other hand, we will use $a \ne b$ for $\neg(a=b)$, $a\nleq b$ for $\neg(a\le b)$, and $a<b$ for $a\le b \land a\ne b$. The archimedean property of $\R$, i.e.,
\[(\forall x\in \R) (\exists n\in\N) (n\ge \abs{x})\]
does not transfer to $\ster\R$, at least, not as the statement that $\ster\R$ has the archimedean property. But
\[(\forall x\in \ster\R) (\exists n\in\ster\N) (n\ge \abs{x})\]
holds, and can be viewed as the transferred statement of the archimedean property.
Also, $\ster\R$ has zero divisors and is not totally ordered.

\begin{exc}
$(\forall x\in\ster\R)$ $(x\in \ster[0,1] \iff 0 \le x\le 1)$.
\end{exc}

\section{The transfer principle}\label{section_transfer}
We will examine in general which statements can be transferred (and in which way). First, we define precisely the kind of statements that we will consider.

\begin{df}
Our formulas are formal expressions containing symbols called \defstyle{variables} (usually denoted by $x,y,z,x_1,x_2,\dots$). Particular kinds of variables are \defstyle{relation variables} (usually denoted by $R, S, R_1, R_2, \dots$) and \defstyle{function variables} (usually denoted by $f, g, f_1, f_2, \dots$).\\
Inductively, \defstyle{terms} are defined by the following rules:
\begin{enumerate}
\item A variable is a term.
\item If $t_1$, \dots, $t_m$ are terms ($m>1$), then also $(t_1, \dots, t_m)$ is a term.
\item If $t$ is a term and $f$ is a function variable, then also $f(t)$ is a term.
\end{enumerate}

The occurrence of a variable $x$ in a formula $P$ is \defstyle{bound} if it occurs in a part of $P$ that is of the form $(\forall x\in t) Q$ or $(\exists x\in t)Q$. Otherwise, the variable $x$ is \defstyle{free} in $P$.

Inductively, \defstyle{formulas} are defined by the following rules:
\begin{enumerate}
\item[F1.] (atomic formulas) If $t_1$, $t_2$ are terms and $R$ is a relation variable, then $t_1=t_2$, $t_1\in t_2$ and $t_1 R\, t_2$ are formulas.
\item[F2.] If $P$, $Q$ are formulas
, then $P\& Q$ is a formula.
\item[F3.] If $P$ is a formula, $x$ is a variable free in $P$ and $t$ is a term in which $x$ does not occur, then $(\exists x\in t) P$ is a formula.
\item[F4.] If $P$ is a formula, $x$ is a variable free in $P$ and $t$ is a term in which $x$ does not occur, then $(\forall x\in t) P$ is a formula.
\item[F5.]\label{rule_impl} If $P$, $Q$ are formulas
, then $P \implies Q$ is a formula.
\item[F6.]\label{rule_neg} If $P$ is a formula, then $\neg P$ is a formula.
\item[F7.]\label{rule_or} If $P$, $Q$ are formulas
, then $P \vee Q$ is a formula.
\end{enumerate}

A \defstyle{sentence} is a formula in which all occurring free variables are substituted by objects, which we call the \defstyle{constants} or \defstyle{parameters} of the sentence. The meaning associated to a sentence is as it occurs in the normal use within mathematics (we will not formalize this; also, we introduce extra brackets in formulas to make clear the precedence of the operations).
\end{df}

\begin{nota}
We denote $t(x_1,\dots, x_m)$ (or shortly $t(x_j)$) for a term $t$ in which the only occurring variables are $x_1$, \dots, $x_m$. We denote by $t(c_1,\dots, c_m)$ (or shortly $t(c_j)$) the term $t$ in which the variable $x_j$ has been substituted by the object $c_j$ (for $j=1,\dots, m$).\\
Similarly, we denote $P(x_1,\dots, x_m)$ (or shortly $P(x_j)$) for a formula $P$ in which the only occurring \emph{free} variables are $x_1$, \dots, $x_m$. We denote by $P(c_1,\dots, c_m)$ (or shortly $P(c_j)$) the formula $P$ in which the variable $x_j$ has been substituted by the object $c_j$ (for $j=1,\dots, m$).
\end{nota}

Now we extend our observation in \S \ref{section_internal_objects} about equality of internal objects to more general formulas:
\begin{df}
A formula $P(x_j)$ 
is called \defstyle{transferrable} if for all internal objects $[c_{j,\eps}]$,
\[P(c_{j,\eps}) \mbox{ is true for small }\eps\]
is equivalent with
\[P([c_{j,\eps}]) \mbox{ is true}.\]
\end{df}

\begin{lemma}\label{term-transfer}
Let $t(x_j)$ be a term
. For internal objects $[c_{j,\eps}]$, we have
\[[t(c_{j,\eps})] = t([c_{j,\eps}]).\]
\end{lemma}
\begin{proof}
1. If $t$ is a variable, this is clear.\\
2. Let $t_1$, \dots, $t_m$ be terms. For a term $(t_1, \dots, t_m)$, we find inductively,
\begin{align*}
[(t_1,\dots, t_m)(c_{j,\eps})] &= [(t_1(c_{j,\eps}), \dots, t_m(c_{j,\eps}))] = ([t_1(c_{j,\eps})], \dots, [t_m(c_{j,\eps})])\\
&= (t_1([c_{j,\eps}]), \dots, t_m([c_{j,\eps}])) = (t_1,\dots, t_m)([c_{j,\eps}]).
\end{align*}
3. Let $t(x_j)$ be a term and $f$ a function variable. For a term $f(t)$, we find inductively,
\[[f(t)(\phi_\eps, c_{j,\eps})] = [\phi_\eps(t(c_{j,\eps}))] = [\phi_\eps]([t(c_{j,\eps})]) = [\phi_\eps](t([c_{j,\eps}])) = f(t)([\phi_\eps], [c_{j,\eps}]).\]
\end{proof}

\begin{prop}\label{prop_up_down_transfer}
Let $P(x_j)$ be a formula formed by applying rules F1--F4 only. Then $P(x_j)$ is transferrable.
\end{prop}
\begin{proof}
F1. For atomic formulas, we have observed this in \S \ref{section_internal_objects} (combined with lemma \ref{term-transfer}). We proceed by induction for more general formulas. We put $c_j:=[c_{j,\eps}]$.\\
F2. For a formula of the form $P(x_j)\&Q(x_j)$, we find inductively,
\begin{align*}
&P(c_j)\&Q(c_j) \mbox{ is true}\\
\iff & P(c_{j,\eps}) \mbox{ is true for small $\eps$, and } Q(c_{j,\eps}) \mbox{ is true for small $\eps$}\\
\iff & P(c_{j,\eps})\&Q(c_{j,\eps}) \mbox{ is true for small }\eps.
\end{align*}
F3. For a formula of the form $(\exists x\in t(x_j))P(x, x_j)$, we find inductively,
\begin{align*}
&(\exists x\in t(c_j))P(x, c_j) \mbox{ is true}\\
\iff &\mbox{there exists } c \in t(c_j) \mbox{ such that } P(c, c_j) \mbox{ is true}\\
\iff &\mbox{there exists $(c_\eps)_\eps$ with } c_\eps\in t(c_{j,\eps}), \mbox{ for small $\eps$}\\
&\mbox{such that } P(c_\eps, c_{j,\eps}) \mbox{ is true for small }\eps\\
\iff &(\exists x\in t(c_{j,\eps})) P(x, c_{j,\eps}) \mbox{ is true for small }\eps.
\end{align*}
F4. For a formula of the form $(\forall x\in t(x_j))P(x,x_j)$, we find inductively,
\begin{align*}
&(\forall x\in t(c_j))P(x,c_j) \mbox{ is true}\\
\iff &\mbox{for each $[c_\eps]$ with } c_\eps\in t(c_{j,\eps}) \mbox{ for small $\eps$}, \quad P([c_\eps],c_j) \mbox{ is true}\\
\iff &\mbox{if } c_\eps\in t(c_{j,\eps}) \mbox{ for small $\eps$, then } P(c_\eps,c_{j,\eps}) \mbox{ is true for small $\eps$}.
\end{align*}
We show that this is still equivalent with: $(\forall x\in t(c_{j,\eps}))P(x,c_{j,\eps})$ is true for small $\eps$.\\
$\Rightarrow$: Suppose that $(\forall \eta)$ $(\exists \eps\le\eta)$ $(\exists x\in t(c_{j,\eps}))$ $\neg P(x,c_{j,\eps})$. Then we can find a decreasing sequence $(\eps_n)_{n\in\N}$ tending to $0$ and $c_{\eps_n}\in t(c_{j,\eps_n})$ such that $\neg P(c_{\eps_n},c_{j,\eps_n})$, $\forall n$. Since $t(c_j)$ is internal, $t(c_j)\ne\emptyset$. Hence we can find $c_\eps\in t(c_{j,\eps})$, for small $\eps\notin\{\eps_n: n\in\N\}$. By assumption, $P(c_\eps,c_{j,\eps})$ is true for small $\eps$, contradicting $\neg P(c_{\eps_n},c_{j,\eps_n})$, $\forall n$.\\
$\Leftarrow$: Let $c_\eps\in t(c_{j,\eps})$, for small $\eps$. Then by assumption, $P(c_\eps, c_{j,\eps})$ for small $\eps$.
\end{proof}

We now obtain the transfer principle as an $\eps$-free version of the previous proposition:
\begin{thm}[Transfer Principle]\label{thm_up-down-transfer}
Let $P(a_1,\dots,a_m)$ be a sentence formed by applying rules F1--F4 only, in which the constants $a_j$ are nongeneralized objects. Then $P(a_1, \dots, a_m)$ is true iff $P(\ster a_1, \dots, \ster a_m)$ is true.
\end{thm}


\begin{ex}\leavevmode
Transfer fails for $P\vee Q$, e.g.\ for the sentence $(\forall x\in\R)$ $( x=0 \vee (\exists y\in\R)(x\cdot y = 1))$ and for $(\forall x,y\in\R)$ $(x\le y \vee y\le x)$.
\end{ex}

For rule F5, we have a transferrable substitute:
\begin{enumerate}
\item[F5'.] $[(\exists x\in t)P]\,\&\, [(\forall x\in t) (P\implies Q)]$.
\end{enumerate}
Since the implication is the part that we want to be able to transfer, we will refer to the condition $(\exists x\in y)P$ as the \defstyle{side condition} for the implication.

\begin{prop}\label{prop_extended_transfer}\leavevmode
Let $P(x_j)$ be a formula formed by applying rules F1--F4 and F5' only. Then $P(x_j)$ is transferrable.
\end{prop}
\begin{proof}
We only have to include rule F5' into the inductive proof of proposition \ref{prop_up_down_transfer}.\\
(1) Let $[(\exists x\in t(c_{j,\eps})) P(x,c_{j,\eps})]$ $\&$ $[(\forall x\in t(c_{j,\eps}))$ $(P(x,c_{j,\eps})\implies Q(x,c_{j,\eps}))]$ hold for small $\eps$. By induction and by proposition \ref{prop_up_down_transfer}, $(\exists x\in t(c_j))P(x,c_j)$. Let $c=[c_\eps]\in t(c_j)$ such that $P(c, c_j)$. By induction, $P(c_\eps, c_{j,\eps})$ for small $\eps$. Then by assumption, $Q(c_\eps, c_{j,\eps})$ for small $\eps$. Hence $Q(c, c_j)$ by induction.\\
(2) let $(\exists x\in t(c_j))P(x,c_j)$ and $(\forall x\in t(c_j))$ $(P(x,c_j)\implies Q(x,c_j))$. Then by induction and by proposition \ref{prop_up_down_transfer}, $(\exists x\in t(c_{j,\eps}))P(x,c_{j,\eps})$, for small $\eps$. Suppose that $(\forall\eta>0)$ $(\exists \eps\le \eta)$ $(\exists x\in t(c_{j,\eps}))$ $(P(x, c_{j,\eps})$ $\&$ $\neg Q(x, c_{j,\eps}))$.
Then we find a decreasing sequence $(\eps_n)_{n\in\N}$ tending to $0$ and $c_{\eps_n}\in t(c_{j,\eps})$ such that $P(c_{\eps_n}, c_{j,\eps_n})$ and $\neg Q(c_{\eps_n}, c_{j,\eps_n})$, $\forall n$. Choose $c_\eps\in t(c_{j,\eps})$ with $P(c_\eps, c_{j,\eps})$ if $\eps\notin\{\eps_n: n\in\N\}$. Then $c:=[c_\eps]\in t(c_j)$ and $P(c,c_j)$ holds by induction. By assumption, $Q(c,c_j)$ holds. By induction, $Q(c_\eps, c_{j,\eps})$ holds for small $\eps$, contradicting $\neg Q(c_{\eps_n}, c_{j,\eps_n})$, $\forall n$.
\end{proof}

\begin{ex}
Often, some information can be transferred from a non-transferrable sentence by reformulating. E.g., the fact that every nonzero element in $\R$ is invertible can also be written as
\[(\forall x\in\R\setminus\{0\}) (\exists y\in \R\setminus\{0\}) (xy=1).\]
Hence, by transfer,
\[(\forall x\in\ster(\R\setminus\{0\})) (\exists y\in \ster(\R\setminus\{0\})) (xy=1).\]
No contradiction results with the fact that $\ster\R$ is not a field, since $\ster(\R\setminus\{0\})$ is the set of those $[x_\eps]\in\ster\R$ with $x_\eps\ne 0$, for small $\eps$, and is a strict subset of $\ster\R\setminus\{0\}$.

One can also obtain some (restricted) information out of a disjunction. E.g., the fact that the order on $\R$ is total can be written as
\[(\forall x,y\in \R) (x\le y \vee y\le x),\]
which is not transferrable. But the equivalent statement
\[(\forall x,y\in \R) (\exists e\in\R) (e^2 = e \ \&\ x e\le y e \ \&\ y(1-e)\le x(1-e))\]
is transferrable. 
\end{ex}

\section{The internal definition principle (I.D.P.)}
We will see that internal sets satisfy a lot of properties which are not shared by arbitrary sets of generalized objects. It is therefore interesting to have an easy sufficient condition to check that a set is internal.
\begin{thm}[Internal Definition Principle]\label{thm_IDP}
Let $P(x,x_j)$ be a transferrable formula. Let $A$, $a_j$ be internal objects. Let $\{x\in A: P(x,a_j)\}\ne\emptyset$. Then $\{x\in A: P(x,a_j)\}$ is internal.\\
Explicitly, if $A=[A_\eps]$ and $a_j=[a_{j,\eps}]$, then $\{x\in A: P(x,a_j)\} = [\{x\in A_\eps: P(x,a_{j,\eps})\}]$.
\end{thm}
\begin{proof}
Let $\{x\in A: P(x,a_j)\}\ne\emptyset$, i.e., $(\exists x\in A)$ $P(x,a_j)$. By transfer, $(\exists x\in A_\eps)$ $P(x,a_{j,\eps})$ holds for small $\eps$. For an internal object $c=[c_\eps]$, we have by transfer,
\begin{align*}
c\in \{x\in A: P(x,a_j)\}
&\iff c\in A \text{ and } P(c,a_j)\\
&\iff c_\eps \in A_\eps \text{ and } P(c_\eps ,a_{j,\eps}), \text{ for small }\eps\\
&\iff c_\eps \in \{x\in A_\eps: P(x, a_{j,\eps})\}, \text{ for small }\eps\\
&\iff c\in [\{x\in A_\eps: P(x, a_{j,\eps})\}],
\end{align*}
where the latter internal set is well-defined since the corresponding net is a net of non-empty sets (for small $\eps$). Further, as $A$ is internal, $A$ has only internal elements. Hence $\{x\in A: P(x,a_j)\}=[\{x\in A_\eps: P(x,a_{j,\eps})\}]$ is internal.
\end{proof}

\begin{cor}\label{cor_star_of_a_set}
Let $P(x,x_j)$ be a transferrable formula with $x$, $x_j$ as only free variables. Let $A$, $a_j$ be nongeneralized objects. If $\{x\in A: P(x,a_j)\}\ne\emptyset$, then
\[\ster\{x\in A: P(x,a_j)\} = \{x\in \ster A: P(x,\ster a_j)\}.\]
\end{cor}
\begin{proof}
By construction of internal sets, if $B=\{x\in A: P(x,a_j)\}$ is a nonempty (nongeneralized) set, then $\ster B$ is also not empty. The proof of the internal definition principle shows that then $\ster\{x\in A: P(x,a_j)\} = [\{x\in A: P(x,a_j)\}] = \{x\in \ster A: P(x, \ster a_j)\}$.
\end{proof}

\begin{exc}
Let $A,B,C$ be non-empty nongeneralized sets.\\
(a) $\ster(A\cap B)=\ster A\cap \ster B$ and $\ster (A\setminus B)\subseteq \ster A\setminus \ster B$.\\
(b) If $A, B\subseteq C$ and $\ster C$ is defined (i.e., the elements of $A$ and $B$ are `of the same type'), then $\ster(A\cup B) = \{a e + b (1-e): a\in A, b\in B, e\in\ster\R, e^2=e\}\supseteq \ster A\cup \ster B$.\\
(c) If $A$, $B$ are internal, then $A\cap B$ is internal or empty.\\
(d) If $A\subseteq B$ and $g$: $B\to C$ is an extension of $f$: $A\to C$, then $\ster g$ is an extension of $\ster f$.
\end{exc}

\section{Saturation and spilling principles}\label{section-saturation}
The principles in the previous sections give us an insight in which properties of generalized objects can be systematically obtained (and in an `$\eps$-free' way), but the properties are often hardly easier obtained than by working directly on the nets.\\The principles in this section will allow for quite some short-cuts in proofs, and will also suggest ways to discover properties that are not so easily guessed directly on the nets.
\begin{df}
A family of sets $(A_i)_{i\in I}$ has the \defstyle{finite intersection property} (F.I.P.) if for each finite subset $F \subseteq I$, $\bigcap_{i\in F}A_i\ne \emptyset$.
\end{df}
\begin{thm}[Saturation Principle]
Let $X$ be an internal set. For each $n\in\N$, let $A_n\subseteq X$ such that $A_n$ or $X\setminus A_n$ is internal. If $(A_n)_{n\in\N}$ has the F.I.P., then $\bigcap_{n\in\N} A_n$ is not empty.
\end{thm}
\begin{proof}
Let $(B_n)_{n\in\N}$, $(X\setminus C_n)_{n\in\N}$ be sequences of internal subsets of $X$ such that $B_1\cap\cdots \cap B_n\cap C_j\ne\emptyset$, for each $n,j\in\N$. It suffices to show that $\bigcap_{n\in\N} (B_n\cap C_n)\ne\emptyset$. Let $B_n=[B_{n,\eps}]$ and $X\setminus C_n=[X_\eps \setminus C_{n,\eps}]$. For $n,j\in\N$ with $j\le n$, let $x_{n,j}\in B_1\cap \cdots \cap B_n\cap C_j$. Since $X$ is internal, also $x_{n,j}=:[x_{n,j,\eps}]$ are internal. Then there exist $\eta_n\in (0,1/n)$ such that $x_{n,j,\eps}\in B_{1,\eps}\cap\cdots\cap B_{n,\eps}$, $\forall\eps\le \eta_n$, $\forall j\le n$. W.l.o.g., $(\eta_n)_{n\in\N}$ is decreasing. For any $a=[a_\eps]\in X$, we have $a\in C_n$ iff $\neg (a\in X\setminus C_n)$ iff $\neg (a_\eps\in X_\eps \setminus C_{n,\eps}$, for small $\eps)$ iff $(\forall \eta\in (0,1))$ $(\exists \eps\le\eta)$ $(a_\eps\in C_{n,\eps})$. Subsequently choose $\eps_{1,1} > \eps_{2,1}>\eps_{2,2} > \cdots > \eps_{n,1}>\eps_{n,2}>\cdots>\eps_{n,n} > \cdots$ ($n\in\N$) with $\eps_{n,j}\in (0, \eta_n)$ and such that $x_{\eps_{n,j}} :=x_{n,j,\eps_{n,j}} \in C_{j,\eps_{n,j}}$. Choose $x_\eps:= x_{n,1,\eps}$, if $\eta_{n+1}<\eps \le \eta_n$ and $\eps\notin\{\eps_{n,j}: n,j\in\N, j\le n\}$. Then for each $n\in\N$, $x_\eps\in B_{n,\eps}$ for small $\eps$, and $(\forall\eta\in (0,1))$ $(\exists \eps\le\eta)$ $(x_\eps\in C_{n,\eps})$. Hence $x:=[x_\eps]\in \bigcap_{n\in\N} (B_n\cap C_n)$.
\end{proof}
\begin{rem}
It is clear from the proof of the saturation principle that, instead of the F.I.P., it is sufficient to assume the slightly weaker property that for each finite number of internal sets $A_{n_1}$, \dots, $A_{n_k}$ and each $A_m$ with $X\setminus A_m$ internal, $A_{n_1}\cap \cdots \cap A_{n_k} \cap A_m\ne \emptyset$. In particular, nonempty cointernal sets have the F.I.P.
\end{rem}
\begin{cor}[Quantifier switching]
Let $X$ be an internal set. For each $n\in\N$, let $P_n(x,x_{n,j})$, $Q_n(x,y_{n,j})$ be transferrable formulas
. Let $a_{n,j}$, $b_{n,j}$ be internal constants. If $P_n$ gets stronger as $n$ increases (i.e., for each $n\in\N$ and $x\in X$, $P_{n+1}(x,a_{n+1,j})\implies P_n(x, a_{n,j})$) and if
\[(\forall n,m\in\N) (\exists x\in X) (P_n(x,a_{n,j}) \,\& \, \neg Q_m(x,b_{m,j})),
\]
then also
\[
(\exists x\in X) (\forall n\in\N) (P_n(x,a_{n,j}) \,\& \, \neg Q_n(x,b_{n,j})).
\]
\end{cor}
\begin{proof}
Let $B_n :=\{x\in X: P_n(x, a_{n,j})\}$ and $C_n :=\{x\in X: \neg Q_n(x, b_{n,j})\}$. By I.D.P., $B_n$, $X\setminus C_n$ are internal or empty. By assumption, $B_n$ are not empty and $B_{n+1}\subseteq B_n$, $\forall n$. If $X\setminus C_n$ is empty, then $C_n = X$, and $C_n$ can be dropped from the sequence. By assumption, for each $n,m\in\N$, $B_1\cap\cdots \cap B_n\cap C_m = B_n\cap C_m\ne\emptyset$. The result follows by (the remark to) the saturation principle.
\end{proof}
Just like the previous corollary, the corollaries known as overspill and underspill, which will soon be formulated, are convenient for practical use.

\begin{df}
Let $a,b\in\ster\R$. Then $a$ is called \defstyle{infinitely large} if $\abs{a}\ge n$, for each $n\in\N$; $a$ is called \defstyle{finite} if $\abs{a}\le N$, for some $N\in\N$; $a$ is called \defstyle{infinitesimal} if $\abs{a}\le 1/n$, for each $n\in\N$.
We denote $a\approx b$ iff $a-b$ is infinitesimal. We denote the set of finite elements of $\ster\R$ by $\Fin(\ster\R)$.
\end{df}

\begin{lemma}\label{lemma_finite}
Let $a\in\ster\R$. If for each infinitely large $m\in\ster\N$, $\abs{a}\le m$, then $a$ is finite.
\end{lemma}
\begin{proof}
Suppose that $a$ is not finite. Then $(\forall n\in\N)$ $(\exists m\in\ster\N)$ $(m\ge n \,\& \abs{a}\nleq m)$. By quantifier switching, there exists $m\in\ster\N$ such that $\abs{a}\nleq m$ and $m\ge n$, for each $n\in\N$, contradicting the hypotheses.
\end{proof}

\begin{thm}[Spilling principles]
Let $A\subseteq\ster\N$ be internal.
\begin{enumerate}
\item (Overspill) If $A$ contains arbitrarily large finite elements (i.e., for each $n\in\N$, there exists $m\in A$ with $m\ge n$), then $A$ contains an infinitely large element.
\item (Underspill) If $A$ contains arbitrarily small infinitely large elements (i.e., for each infinitely large $\omega\in\ster\N$, there exists $a\in A$ with $a\le \omega$), then $A$ contains a finite element.
\item (Overspill) If $\N\subseteq A$, then there exists an infinitely large $\omega\in\ster\N$ such that $\{n\in\ster\N: n\le \omega\}\subseteq A$.
\item (Underspill) If $A$ contains all infinitely large elements of $\ster\N$, then $A\cap \N\ne\emptyset$.
\end{enumerate}
\end{thm}
\begin{proof}
1. As $(\forall n\in\N)$ $(\exists m\in A)$ $(m\ge n)$, there exists an infinitely large $m\in A$ by quantifier switching.

2. By transfer on the sentence
\[(\forall X\in\Powerset(\N)) (\exists m\in X) (\forall n\in X) (n\ge m),\]
every internal subset of $\ster\N$ has a smallest element. Let $n_{min}$ be the smallest element of $A$. Then $n_{min}\le \omega$, for each infinitely large $\omega\in\ster\N$. By lemma \ref{lemma_finite}, $n_{min}$ is finite.

3. First, let $n_0\in\N$. By transfer on the sentence
\[(\forall X\in \Powerset(\N)) [(1\in X \ \&\ \dots \ \&\ n_0\in X) \implies (\forall m\in\N) (m\le n_0\implies m\in X)]\]
(side conditions are trivially fulfilled), any internal subset of $\ster\N$ that contains $\N$ also contains $\{m\in\ster\N: m\le n_0\}$, for any $n_0\in\N$. Then
\[
B =\{n\in\ster\N: (\forall m\in\ster\N) (m\le n \implies m\in A)\}.
\]
is internal by I.D.P.\ (since the side condition is trivially fulfilled and $B\ne\emptyset$) and contains $\N$. By part 1, $B$ contains an infinitely large $\omega$. Hence $\{n\in\ster\N: n\le\omega\}\subseteq A$.

4. Let
\[
B =\{n\in\ster\N: (\forall m\in\ster\N) (m\ge n \implies m\in A)\}.
\]
By I.D.P., $B$ is internal (since the side condition is trivially fulfilled and $B\ne\emptyset$). By part~2, $B$ contains a finite element, i.e., there exists $n\in B$ and $N\in\N$ such that $n\le N$. By definition of $B$, $N\in A$.
\end{proof}

\begin{cor}
$\N$ and $\Fin(\ster\R)$ are external subsets of $\ster\R$.
\end{cor}

\begin{cor}[Rigidity]
Let $f$, $g$ be internal maps $\ster\R\to\ster\R$. If $f(x)=g(x)$ for each $x\approx 0$, then there exists $r\in\R^+$ such that $f(x)=g(x)$ for $x\in\ster\R$ with $\abs x\le r$.
\end{cor}
\begin{proof}
By underspill on $\{n\in\ster\N: (\forall x\in\ster\R) (\abs x\le 1/n \implies f(x) = g(x))\}$.
\end{proof}

\section{Calculus on $\ster\R$: examples}
By transfer, many concepts defined for nongeneralized objects have a counterpart for internal generalized objects. As illustrated below, we can often characterize the cor\-res\-pon\-ding concept by a property that can also be defined for external (=non-internal) generalized objects. This yields an intrinsic development of the theory, without reference to the structure of the internal objects as nets.\\
Stated otherwise: generalized objects are judged by their properties (in the formal language), which are often similar to those of nongeneralized objects (by transfer), rather than viewed as nets of nongeneralized objects that are `wildly moving around'.
\vskip 8pt plus1pt

If $\mathcal B$ is the set of all non-empty bounded subsets of $\R$, then
\[\ster{\mathcal B}=\{A\in\ster\Powerset(\R): (\exists R\in\ster\R) (\forall x\in A) (\abs x\le R)\}\]
by corollary \ref{cor_star_of_a_set}.
\begin{df}
A subset $A$ of $\ster\R$ is \defstyle{$*$-bounded} if $(\exists R\in\ster\R) (\forall x\in A) (\abs x\le R)$.
\end{df}
Hence $\ster{\mathcal B}$ is the set of all internal $*$-bounded subsets of $\ster\R$. A nonempty subset $A\subseteq\R$ is bounded iff $\ster A$ is $*$-bounded.
\vskip 8pt plus1pt

If $\mathcal F$ is the set of all non-empty closed subsets of $\R$, then
\[\ster{\mathcal F}=\{A\in\ster\Powerset(\R): (\forall x\in\ster\R) [(\forall r\in\ster(\R^+)) (\exists a\in A) (\abs{x-a}\le r)\implies x\in A]\}
\]
by corollary \ref{cor_star_of_a_set} (since the side-condition $(\exists x\in\ster\R)(\forall r\in\ster(\R^+)) (\exists a\in A) (\abs{x-a}\le r)$ is always fulfilled, and thus becomes redundant).

\begin{df}
A subset $A$ of $\ster\R$ is \defstyle{$*$-closed} if every $x\in\ster\R$ with the property that $(\forall r\in\ster(\R^+)) (\exists a\in A) (\abs{x-a}\le r)$ belongs to $A$.
\end{df}
\vskip 8pt plus1pt

If $\mathcal K$ is the set of all non-empty compact subsets of $\R$, then $\mathcal K = \mathcal B\cap \mathcal F$, so $\ster{\mathcal K}= \ster{\mathcal B}\cap \ster{\mathcal F}$.

\begin{df}
A subset $A$ of $\ster\R$ is \defstyle{$*$-compact} if $A$ is $*$-bounded and $*$-closed.
\end{df}

The map $\max$: $\mathcal K\to\R$ is well-defined. Hence $\stermathop\max$: $\ster{\mathcal K}\to \ster\R$ is well-defined. Since
\[(\forall K\in\mathcal K)(\forall x\in \R)(x=\max(K) \iff x\in K \ \&\ (\forall y\in K) (x\ge y)),\]
we see that $\stermathop\max(K)$ is the maximum of $K$ for the usual order on $\ster\R$ (by transfer).
\vskip 8pt plus1pt

Let $A\subseteq\R$. Let $\Cnt(A)$ be the set of all continuous maps $A\to\R$. Then
\begin{multline*}
\ster\Cnt(A)=\{f\in\ster{(\R^A)}:\\
(\forall x\in \ster A) (\forall r\in\ster(\R^+)) (\exists \delta\in\ster(\R^+)) (\forall y\in\ster A) (\abs{x-y}\le\delta \implies \abs{f(x)-f(y)}\le r)\}
\end{multline*}
by corollary \ref{cor_star_of_a_set}.

\begin{df}
Let $\emptyset\ne A\subseteq \ster\R$. A map $f$: $A\to\ster\R$ is called \defstyle{$*$-continuous} if $(\forall x\in A) (\forall r\in\ster(\R^+)) (\exists \delta\in\ster(\R^+)) (\forall y\in A) (\abs{x-y}\le\delta \implies \abs{f(x)-f(y)}\le r)$.
\end{df}

\begin{prop}
Let $K\subseteq \ster\R$ be internal and $*$-compact. Let $f$ be an internal $*$-continuous map $K\to\ster\R$. Then $f(K)$ is $*$-compact. In particular, $f$ reaches a maximum on $K$.
\end{prop}
\begin{proof}
We would like to apply transfer to $(\forall K\in \mathcal K) (\forall f\in\Cnt(K)) (f(K)\in\mathcal K)$. Then we have to consider $\Cnt$ as a map $\Powerset(\R)\to \Powerset(\mathcal F(\R,\R))$: $A\mapsto \Cnt(A)$, where we denote by $\mathcal F(\R, \R)$ the set of all (partially defined) functions $\R\to\R$. Then we obtain the transferred property $(\forall K\in \ster{\mathcal K}) (\forall f\in(\ster\Cnt)(K)) (f(K)\in\ster{\mathcal K})$. By transfer on
\begin{multline*}
(\forall X\in\Powerset(\R)) \big(\forall f\in\mathcal F(\R,\R)) (f\in \Cnt(X) \iff\\
\begin{cases}
(\forall x\in X) (\exists y\in \R) (f(x) = y)\\
(\forall x\in X) (\forall r\in \R^+) (\exists \delta\in\R^+) (\forall y\in X) (\abs{x-y}\le\delta \implies \abs{f(x)-f(y)}\le r).
\end{cases}
\end{multline*}
we see that for internal $A\subseteq\ster\R$, $(\ster\Cnt)(A)$ is the set of all internal functions that are defined and $*$-continuous on $A$.
\end{proof}

\begin{df}
Let $\Omega\subseteq \R^d$ be open. Then we denote $\ster\Omega_c:= \bigcup_{K\csub\Omega}\ster K$.
\end{df}

\begin{prop}[Infinitesimal characterization of continuity]\label{prop_cnt_char}
Let $f$: $\Omega\to\C$. The following are equivalent:
\begin{enumerate}
\item $f$ is continuous
\item $(\forall x,y\in\ster\Omega_c)$ $(x\approx y\implies f(x)\approx f(y))$.
\end{enumerate}
\end{prop}
\begin{proof}
$\Rightarrow$: let $x\in\ster K$, $K\csub\Omega$ and $y\in \ster L$, $L\csub\Omega$ with $x\approx y$. Let $r\in\R^+$. As $f$ is uniformly continuous on $K\cup L\csub \Omega$, there exists $\delta\in\R^+$ such that $(\forall x', y'\in K\cup L)$ $(\abs{x'-y'}\le\delta \implies \abs{f(x')-f(y')}\le r)$. By transfer, $(\forall x',y'\in\ster (K\cup L))$ $(\abs{x'-y'}\le\delta \implies \abs{f(x')-f(y')}\le r)$. As $\ster K\cup\ster L\subseteq \ster (K\cup L)$, $\abs{f(x)-f(y)}\le r$. Since $r\in\R^+$ arbitrary, $f(x)\approx f(y)$.\\
$\Leftarrow$: let $x\in\Omega$ and $r\in\R^+$. Let
\[A=\{n\in\ster\N: (\forall y\in\ster\Omega) (\abs{x-y}\le 1/n\implies \abs{f(x)-f(y)}\le r)\}.\]
If $n\in\ster\N$ is infinitely large and $\abs{x-y}\le 1/n$, then $y\in \ster\Omega_c$ and $x\approx y$, so $n\in A$ by assumption. Further, $A$ is internal by I.D.P.\ (since the side condition for the implication is always fulfilled). By underspill, $A$ contains some $n\in\N$.
\end{proof}
\vskip 8pt plus1pt minus2pt

Since $\partial_j$: $\Cnt[1](\Omega)\to\Cnt(\Omega)$, we have for $f\in\ster{\Cnt[1]}(\Omega)$ that $(\ster\partial_j) f\in\ster\Cnt(\Omega)$ and $(\forall x\in\ster\Omega) (\forall r\in\ster(\R^+))$ $(\exists \delta\in\ster(\R^+))$ $(\forall h\in\ster\R)$ $(0\, \ster\!\!<\abs h\le \delta \implies \abs[\big]{\frac{f(x+he_j)-f(x)}{h} - (\ster\partial_j) f(x)}\le r)$ by transfer. (The inverse of $h$ is defined if $\abs h \, \ster \!\!> 0$.) We can again define the concept of a $*$-partial derivative for any map $f$: $\ster\Omega\to\ster\C$. Also for the (differential) algebraic operations on functions, we will drop stars and simply write $\partial_j f$ instead of $(\ster\partial_j) f$.

\section{Colombeau generalized objects: examples}
\subsection*{$\GenR$ and the $\caninf$-topology on $\ster\R$}
\begin{df}
We denote $\caninf:=[\eps]\in \ster\R$.
We call \defstyle{$\caninf$-topology} on $\ster\R^d$ the translation invariant topology with $\{B(0,\caninf^m): m\in\N\}$ as a local base of neighbourhoods of $0$ (with $B(a,r):=\{x\in\ster\R^d: \abs {x-a} < r\}$, for $a\in\ster\R^d$ and $r\in\ster(\R^+)$).\\
We call $x\in\ster\R^d$ \defstyle{negligible} if $\abs x\le\caninf^m$, for each $m\in\N$ (i.e., if $x$ belongs to the intersection of all $\caninf$-neighbourhoods of $0$). For $x,y\in\ster\R^d$, we write $x\approxeq y$ if $x-y$ is negligible.
We call $x\in\ster\R^d$ \defstyle{moderate} if there exists $N\in\N$ such that $\abs x\le\caninf^{-N}$. We write $\Mod_{\R^d}$ for the set of moderate elements and $\Null_{\R^d}$ for the set of negligible elements.
\end{df}

We have a similar characterization for $\caninf$-continuity (=continuity in the $\caninf$-topology) as in proposition \ref{prop_cnt_char}:

\begin{prop}\label{rho-cnt-sterR}
Let $f$: $\ster\Omega\to\ster\C$ be internal and $a\in\ster\Omega$. The following are equivalent:
\begin{enumerate}
\item $(\forall m\in\N)$ $(\exists n\in\N)$ $(\forall x\in\ster\Omega)$ $(\abs{x-a}\le \rho^n\implies \abs{f(x) - f(a)}\le \rho^m)$
\item $(\forall x\in \ster\Omega)$ $(x\approxeq a\implies f(x)\approxeq f(a))$.
\end{enumerate}
\end{prop}
\begin{proof}
$\Rightarrow$: Let $x\in\ster\Omega$ with $x\approxeq a$. Let $m\in\N$. By assumption, $\abs{f(x)-f(a)}\le\rho^m$. As $m\in\N$ is arbitrary, $f(x)\approxeq f(a)$.\\
$\Leftarrow$: let $m\in\N$. Consider
\[
A:=\{n\in\ster\N: (\forall x\in\ster\Omega) (\abs{x-a}\le\rho^n\implies \abs{f(x)-f(a)}\le\rho^m)\}.
\]
By assumption, $A$ contains all infinitely large $n\in\ster\N$. By I.D.P., $A$ is internal (as the side condition for the implication is always fulfilled). By underspill, $A\cap\N\ne\emptyset$.
\end{proof}

Since we are interested in nonlinear operations for generalized functions, we notice that, although the product is not $\caninf$-continuous on the whole space, we have:
\begin{prop}\label{product-rho-cnt}
The product is $\caninf$-continuous on moderate elements.
\end{prop}
\begin{proof}
Since the product $\ster\R^2\to \ster\R$ is internal (it equals $\ster\cdot$, where $\cdot$: $\R^2\to\R$), continuity at $(a,b)\in\ster\R^2$ means that
\[(\forall x,y\in\ster\R) \Big(\left.\begin{matrix}x\approxeq a\\y\approxeq b\end{matrix}\right\}\implies xy\approxeq ab\Big).\]
If $a,b\in \Mod_\R$, then $(x-a)b\in\Null_\R$, so $xb\approxeq ab$. Similarly, $xy\approxeq xb \approxeq ab$.
\end{proof}

Given a non-Hausdorff translation-invariant topology, one obtains a Hausdorff topological space by dividing out the intersection of all neighbourhoods of $0$. This motivates the following definition:
\begin{df}
The ring of \defstyle{Colombeau generalized (real) numbers} is
\[\GenR:= \Mod_\R/\Null_\R.\]
The \defstyle{sharp topology} on $\GenR$ is the Hausdorff (even metrizable) topology induced by the $\caninf$-topology on $\ster\R$. By proposition \ref{product-rho-cnt}, the product is well-defined and continuous on $\GenR$. In fact, $\GenR$ is a topological ring.
\end{df}

\begin{rem}
This definition coincides (up to an isomorphism in a strong sense) with the classical definition
\begin{multline*}
\GenR:= \{(x_\eps)_\eps\in \R^{(0,1)}: (\exists N\in\N)(\abs{x_\eps}\le \eps^{-N} \text{ for small }\eps)\}\\
/\{(x_\eps)_\eps\in \R^{(0,1)}: (\forall m\in\N)(\abs{x_\eps}\le \eps^{m} \text{ for small }\eps)\}
\end{multline*}
\end{rem}
since the only difference with the classical definition is that we have done the identification up to neglibility in two steps (in the first step only identifying up to small $\eps$).
\smallskip

Also in $\GenR^d$, internal sets can be defined:
\begin{df}
Let $\emptyset\ne A_\eps\subseteq \ster\R^d$ for each $\eps$. We denote the equivalence class of $(x_\eps)_\eps$ in $\GenR^d$ again by $[x_\eps]$. Then
\[\{[x_\eps]\in\GenR^d: x_\eps \in A_\eps \text{ for small }\eps\}\]
is the internal subset of $\GenR^d$ with representative $(A_\eps)_\eps$. Equivalently, if $A\subseteq \ster\R^d$ is internal, then, denoting by $[x]= x + \Null_{\R^d}$ the equivalence class of $x\in\ster\R^d$ in $\GenR^d$,
\[\{[x] \in\GenR^d: x\in A\}\]
is the internal subset of $\GenR^d$ with representative $A$. 
\end{df}
The disadvantage of internal sets in $\GenR^d$ (compared to $\ster\R^d$) is that they are not closed under as many operations as the internal sets in $\ster\R^d$. Even $\{x\in A: x\ge 0\} = A\cap [ [0,\infty) ]$ need not be internal if $A\subseteq\GenR$ is internal \cite{OVInternal}. In particular, the analogous statement of the I.D.P.\ does not hold for internal sets in $\GenR^d$. This makes it hard to convert the proof techniques from section \ref{section-saturation} to techniques for internal sets in $\GenR^d$. Therefore, it is often advantageous to use internal sets in $\ster\R^d$ to prove statements about internal sets in $\GenR^d$.
\smallskip

Internal sets can sometimes compensate for the fact that $\GenR^d$ is not locally compact:
\begin{prop}
Let $A\subseteq \GenR^d$ be internal and sharply bounded and let $B\subseteq\GenR^d$ be an internal sharp neighbourhood of $A$. Then there exists $M\in\N$ such that for each $a\in A$, $B(a,\caninf^M)=\{x\in\GenR^d: \abs{x-a}<\caninf^M\} \subseteq B$.
\end{prop}
\begin{proof}
Let $\bar A$, $\bar B\subseteq \ster\R^d$ be representatives of $A$, $B$ (with $\bar A$ sharply bounded). Let $\tilde x\in\GenR^d\setminus B$ with representative $x\in \ster\R^d$, then $\ster d(x,\bar B)\not\approxeq 0$. Suppose that the conclusion does not hold. Then $(\forall n\in\N)$ $(\exists a\in \bar A)$ $(\exists x\in B(a,\caninf^n))$ $(\ster d(x, \bar B)\not\approxeq 0)$.
Thus we find $k_n\in\N$, $\forall n\in\N$, such that $(\forall n\in\N)$ $(\exists a\in \bar A)$ $\neg (\forall x\in B(a,\caninf^n))$ $(\ster d(x, \bar B)\le \caninf^{k_n})$. By quantifier switching, we would find $a\in \bar A$ such that for each $n\in\N$, $B(a,\caninf^n)$ contains some $x\in\ster\R^d$ for which $\ster d(x, \bar B)\not\approxeq 0$. Since $\bar A$ is sharply bounded, we find $\tilde a\in A$ such that $B(\tilde a, \caninf^n)\not\subseteq B$, for each $n\in\N$, contradicting the fact that $B$ is a sharp neighbourhood of $A$.
\end{proof}
If $B$ is not internal, the previous proposition fails in general. E.g., let $A=[0,1]\sptilde\times\{0\}\subseteq \GenR^2$ and let $B=\bigcup_{n\in\N, \tilde x\approx 1/n} B((\tilde x, 0), \caninf^n) \cup \bigcup_{\tilde x\in\widetilde{[0,1]}, \tilde x\not\approx 1/n,\forall n} B((\tilde x,0),\caninf)\subseteq\GenR^2$. Then $B$ is a sharp neighbourhood of $A$, but $(1/n, \caninf^n)\notin B$, for each $n\in\N$.

\subsection*{$\Gen(\Omega)$ and the $\rho$-topology on $\ster{\Cnt[\infty](\Omega)}$}
Let $\Omega\subseteq\R^d$ be open. For $u\in\Cnt[\infty](\Omega)$, let $p_m(u):= \sup_{x\in K_m, \abs{\alpha}\le m} \abs{\partial^\alpha u(x)}$, where $(K_m)_m$ is a compact exhaustion of $\Omega$ (the seminorms $p_m$ describe the usual locally convex topology on $\Cnt[\infty](\Omega)$).
\begin{df}
We call \defstyle{$\caninf$-topology} on $\ster{\Cnt[\infty]}(\Omega)$ the translation invariant topology with $\{B_m(0,\caninf^m): m\in\N\}$ as a local base of neighbourhoods of $0$ (with $B_m(0,r):=\{u\in\ster{\Cnt[\infty]}(\Omega): \ster p_m(u) < r\}$, for $r\in\ster(\R^+)$).\\
We call $u\in\ster{\Cnt[\infty](\Omega)}$ \defstyle{moderate} (resp.\ \defstyle{negligible}) if $\ster p_m(u)$ is moderate (resp.\ negligible) in $\ster\R$, for each $m\in\N$. Again, $u$ belongs to the intersection of all $\caninf$-neighbourhoods of $0$ iff $u$ is negligible. We write $u\approxeq_{\Cnt[\infty](\Omega)} v$ (or $u\approxeq v$ if the space is clear from the context) if $u-v$ is negligible. We write $\Mod_{\Cnt[\infty](\Omega)}$ for the set of moderate elements and $\Null_{\Cnt[\infty](\Omega)}$ for the set of negligible elements in $\ster{\Cnt[\infty]}(\Omega)$.
\end{df}
Explicitly,
\begin{align*}
\Mod_{\Cnt[\infty](\Omega)} &= \{u\in\ster{\Cnt[\infty]}(\Omega):  (\forall \alpha\in\N^d) (\forall K\csub \Omega) (\max_{x\in\ster K}\abs{\partial^\alpha u(x)} \text{ is moderate})\}\\
&=\{u\in\ster{\Cnt[\infty]}(\Omega): (\forall \alpha\in\N^d) (\forall x\in \ster\Omega_c) (\partial^\alpha u(x) \text{ is moderate})\}
\end{align*}
\begin{align*}
\Null_{\Cnt[\infty](\Omega)} &= \{u\in\ster{\Cnt[\infty]}(\Omega):  (\forall \alpha\in\N^d) (\forall K\csub \Omega) (\max_{x\in\ster K}\abs{\partial^\alpha u(x)} \text{ is negligible})\}\\
&=\{u\in\ster{\Cnt[\infty]}(\Omega): (\forall \alpha\in\N^d) (\forall x\in \ster\Omega_c) (\partial^\alpha u(x) \text{ is negligible})\}.
\end{align*}

\begin{prop}\label{rho-cnt-Cnt-infty}
Let $T$: $\ster{\Cnt[\infty]}(\Omega)\to\ster{\Cnt[\infty]}(\Omega)$ be internal and $u\in\ster{\Cnt[\infty]}(\Omega)$. Then the following are equivalent:
\begin{enumerate}
\item $(\forall m\in\N)$ $(\exists n\in\N)$ $(\forall v\in\ster{\Cnt[\infty]}(\Omega))$ $(\ster p_n(v-u)\le \rho^n\implies \ster p_m(T(v) - T(u))\le \rho^m)$
\item $(\forall v\in \ster{\Cnt[\infty]}(\Omega))$ $(v\approxeq u\implies T(v)\approxeq T(u))$.
\end{enumerate}
\end{prop}
\begin{proof}
Analogous to the proof of proposition \ref{rho-cnt-sterR}.
\end{proof}

For similar reasons as on $\ster\R$, the algebra of \defstyle{Colombeau generalized functions} on $\Omega$ is
$\Gen(\Omega) := \Mod_{\Cnt[\infty](\Omega)}/\Null_{\Cnt[\infty](\Omega)}$. 
The correspondence of this definition with the classical definition of $\Gen(\Omega)$ follows from $\stermathop\max_{x\in\ster K}\abs{\partial^\alpha u(x)} = [\max_{x\in K} \abs{\partial^\alpha u_\eps(x)}]$. Explicitly, by I.D.P.,
\begin{multline*}
\{\abs{\partial^\alpha u(x)}: x\in\ster K\} = \{y\in\ster\R: (\exists x\in \ster K) (y=\abs{\partial^\alpha u(x)})\}\\
= [\{y\in\R: (\exists x\in K) (y=\abs{\partial^\alpha u_\eps(x)})\}] = [\{\abs{\partial^\alpha u_\eps(x)}: x\in K\}].
\end{multline*}
The sharp topology is the Hausdorff (even metrizable) topology on $\Gen(\Omega)$ induced by the $\caninf$-topology on $\ster{\Cnt[\infty]}(\Omega)$. Again, well-definedness of internal operations (such as the product) on $\Gen(\Omega)$ corresponds with $\caninf$-continuity of the corresponding operations in $\ster{\Cnt[\infty]}(\Omega)$.

\begin{prop}[Automatic continuity]
Let $T$: $\Gen(\Omega)\to\Gen(\Omega)$ be an internal operator. Then $T$ is sharply continuous.
\end{prop}
\begin{proof}
To be precise, if $T$ has $\bar T$: $\ster{\Cnt[\infty]}(\Omega)\to \ster{\Cnt[\infty]}(\Omega)$ as a representative, we assume that $T$ is well-defined on the whole of $\Gen(\Omega)$, i.e., $\bar T(\Mod_{\Cnt[\infty](\Omega)})\subseteq \Mod_{\Cnt[\infty](\Omega)}$ and $u\approxeq v\implies \bar T u \approxeq \bar T v$, for each $u,v\in \Mod_{\Cnt[\infty](\Omega)}$. By proposition \ref{rho-cnt-Cnt-infty}, this means that $\bar T$ is $\caninf$-continuous on $\Mod_{\Cnt[\infty](\Omega)}$.
\end{proof}
As illustrated by the previous theorem, Colombeau theory can from the nonstandard point of view be considered as the study of $\caninf$-continuous internal maps defined on the (external) set of moderate elements (in $\ster\R$, $\ster{\Cnt[\infty]}(\Omega)$, \dots).

\subsection*{$\Gen^\infty$-regularity}
\begin{df}
The subalgebra of $\Gen^\infty$-regular Colombeau generalized functions on $\Omega$ is defined by
\[
\Gen^\infty(\Omega) := \{u\in\ster{\Cnt[\infty]}(\Omega): (\forall K\csub \Omega) (\exists N\in\N) (\forall \alpha\in\N^d) (\max_{x\in\ster K}\abs{\partial^\alpha u(x)} \le \caninf^{-N})\}/\Null_{\Cnt[\infty](\Omega)}.
\]
\end{df}

\begin{df}
Let $u\in\ster{\Cnt[\infty]}(\Omega)$ and $x\in\ster\Omega$. We say that $u$ is $\Gen^\infty$-regular at $x$ if there exists $N\in\N$ such that for each $\alpha\in\N^d$, $\abs{\partial^\alpha u(x)}\le \caninf^{-N}$.
\end{df}

\begin{prop}[Pointwise characterization of $\Gen^\infty(\Omega)$]\label{prop_pointwise_regularity}
Let $u\in\ster{\Cnt[\infty]}(\Omega)$. The following are equivalent:
\begin{enumerate}
\item $(\forall K\csub \Omega)$ $(\exists N\in\N)$ $(\forall \alpha\in\N^d)$ $(\max_{x\in\ster K}\abs{\partial^\alpha u(x)} \le \caninf^{-N})$
\item $u$ is $\Gen^\infty$-regular at each $x\in\ster\Omega_c$.
\end{enumerate}
\end{prop}
\begin{proof}
$\Rightarrow$: clear.\\
$\Leftarrow$: Suppose that (1) does not hold. Then we find $K\csub \Omega$ and $\alpha_n\in\N^d$, $\forall n\in\N$ such that $(\forall n\in\N)$ $(\exists x\in\ster K)$ $(\abs{\partial^{\alpha_n} u(x)}\nleq\caninf^{-n})$. By quantifier switching, $(\exists x\in\ster K)$ $(\forall n\in\N)$ $(\abs{\partial^{\alpha_n} u(x)}\nleq\caninf^{-n})$, contradicting the hypotheses.
\end{proof}

Hence we obtain (cf.\ \cite[Thm.~5.1]{OPS}):
\begin{align*}
\Gen^\infty(\Omega) = & \{u\in\ster{\Cnt[\infty]}(\Omega): u \text{ is $\Gen^\infty$-regular at each } x\in\ster\Omega_c\}\\
&/\{u\in\ster{\Cnt[\infty]}(\Omega):  (\forall \alpha\in\N^d) (\forall x\in\ster\Omega_c) (\partial^\alpha u(x) \text{ is negligible})\}.
\end{align*}

Similarly, we have the following refinement (cf.\ \cite[Prop.~5.3]{HVPointwiseReg}):
\begin{prop}\label{prop_pointwise_reg_local}
Let $u\in\ster{\Cnt[\infty]}(\Omega)$ and $(A_n)_{n\in\N}$ a decreasing sequence of internal subsets of $\ster\Omega$. Let $B:=\bigcap_{n\in\N} A_n$.
Then the following are equivalent:
\begin{enumerate}
\item $(\forall K\csub \Omega)$ $(\exists N\in\N)$ $(\forall \alpha\in\N^d)$ $(\exists m\in\N)$ $(\forall x\in\ster K\cap A_m)$ $(\abs{\partial^\alpha u(x)} \le \caninf^{-N})$
\item $u$ is $\Gen^\infty$-regular at each $x\in B\cap \ster\Omega_c$.
\end{enumerate}
\end{prop}
\begin{proof}
$\Rightarrow$: clear.\\
$\Leftarrow$: Suppose that (1) does not hold. Then we find $K\csub\Omega$ and $\alpha_n\in\N^d$, $\forall n\in\N$ such that $(\forall n,m\in\N)$ $(\exists x\in\ster K)$ $(x\in A_m \,\&\, \abs{\partial^{\alpha_n} u(x)}\nleq\caninf^{-n})$. By quantifier switching, $(\exists x\in\ster K)$ $(\forall n\in\N)$ $(x\in A_n \,\&\, \abs{\partial^{\alpha_n} u(x)}\nleq\caninf^{-n})$, contradicting the hypotheses.
\end{proof}

If we want to translate this result into the language of internal sets in Colombeau theory \cite{OVInternal}, we still have to ensure independence of representatives:
\begin{cor}
Let $u\in\Gen(\Omega)$ and $(A_n)_{n\in\N}$ a decreasing sequence of internal subsets of $\widetilde\Omega$. Let $B:= \bigcap_{n\in\N} A_n$. Suppose that $B\cap\widetilde\Omega_c\ne\emptyset$. Then the following are equivalent:
\begin{enumerate}
\item $(\forall K\csub \Omega)$ $(\exists N\in\N)$ $(\forall \alpha\in\N^d)$ $(\exists m\in\N)$ $(\forall \tilde x\in \widetilde K\cap A_m)$ $(\abs{\partial^\alpha u(\tilde x)} \le \caninf^{-N})$
\item $u$ is $\Gen^\infty$-regular at each $\tilde x\in B\cap \widetilde\Omega_c$ (i.e., $(\forall \tilde x\in B\cap\widetilde\Omega_c)$ $(\exists N\in\N)$ $(\forall \alpha\in\N^d)$ $(\abs{\partial^\alpha u(\tilde x)}\le\caninf^{-N})$).
\end{enumerate}
\end{cor}
\begin{proof}
$\Rightarrow$: clear.\\
$\Leftarrow$: Let $\bar A_n$ be representatives of $A_n$. Let $\bar C_n:= \bar A_n + \caninf^n = \{x\in \ster\Omega: \ster d(x, \bar A_n)\le \caninf^n\}$. Then $\bar C_n$ are internal by I.D.P.\ and $(\bar C_n)_{n\in\N}$ is decreasing. Let $\tilde x\in B\cap\widetilde\Omega_c$ with representative $x$. Then $x\in \bigcap_{n\in\N} \bar C_n\cap \ster\Omega_c$. Conversely, if $x\in\bigcap_{n\in\N} \bar C_n\cap \ster\Omega_c$, then $x$ represents $\tilde x\in B \cap \widetilde\Omega_c$. The result follows by proposition \ref{prop_pointwise_reg_local}, since for each $K\csub \Omega$, there exists $L\csub\Omega$ such that each representative of $\tilde x\in \widetilde K$ belongs to $\ster L$.
\end{proof}

\subsection*{$\Gen_E$}
Let $E$ be a locally convex vector space (belonging to the nongeneralized objects) with its topology generated by a family of seminorms $(p_i)_{i\in I}$. Then $\ster p_i$: $\ster E\to \ster\R$ are well-defined. As before,
\begin{align*}
\Gen_E \cong & \{ u\in \ster E: (\forall i\in I) (\ster p_i(u) \text{ is moderate})\}\\
&/\{ u\in \ster E: (\forall i\in I) (\ster p_i(u) \text{ is negligible})\}.
\end{align*}

\begin{prop}
Let the topology of $E$ be generated by a countable family of seminorms $(p_n)_{n\in\N}$. Then $\Gen_E$ is complete.
\end{prop}
\begin{proof}
W.l.o.g., $(p_n)_{n\in\N}$ is increasing. Let $(u_n)_{n\in\N}$ be a Cauchy sequence in $\Gen_E$. Let $\bar u_n\in\ster E$ be representatives of $u_n$. Then for each $m\in\N$, there exists $N_m$ (w.l.o.g.\ increasing) such that $\ster p_m(\bar u_k - \bar u_l)\le \caninf^m$, as soon as $k,l\ge N_m$. Hence $(\forall m\in\N)$ $(\exists \bar u \in \ster E)$ $(\ster p_1(\bar u - \bar u_{N_1})\le \caninf \,\&\, \dots \,\&\, \ster p_m(\bar u - \bar u_{N_m})\le \caninf^m)$. By quantifier switching, we find $\bar u \in \ster E$ such that $\ster p_m(\bar u - \bar u_{N_m})\le \caninf^m$, for each $m\in\N$. Then $\ster p_m(\bar u)\le \ster p_m(\bar u_{N_m})+ \caninf^m$ is moderate, for each $m\in\N$, so $\bar u$ represents $u\in\Gen_E$. Let $n\in\N$. Then for each $m\ge n$, $\ster p_n(\bar u - \bar u_{N_m})\le \ster p_m(\bar u - \bar u_{N_m})\le \caninf^m$. Hence $u = \lim_{m\to\infty} u_{N_m}$. Since $(u_n)_{n\in\N}$ is a Cauchy sequence, also $u = \lim_{m\to\infty} u_m$.
\end{proof}

\section{Nonstandard Analysis}
The above construction is (up to details) the one introduced in \cite{SL} as a rigorous model for doing analysis with infinitesimals. The relation between nonlinear generalized functions and this theory was already noticed in \cite{Ob92}. \emph{Nonstandard analysis} is a refinement of \cite{SL}: here the nets are identified in a more sophisticated way than just `for small $\eps$'. If we write $\Filter := \{S\subseteq (0,1): (\exists \eta\in (0,1)) ((0,\eta)\subseteq S)\}$, then `$a_\eps = b_\eps$ for small $\eps$' is equivalent with: $\{\eps\in(0,1): a_\eps = b_\eps \}\in \Filter$. The set $\Filter$ clearly has the following set-theoretic properties:
\begin{enumerate}
\item[(F1)] $(0,1)\in\Filter$
\item[(F2)] $S\in\Filter$, $S\subseteq T\subseteq (0,1)$ $\implies$ $T\in\Filter$
\item[(F3)] $S,T\in\Filter \implies S\cap T\in\Filter$.
\item[(F4)] $\bigcap_{S\in\mathcal  F} S = \emptyset$.
\end{enumerate}
A set with these properties is called a \emph{free filter} on $(0,1)$.\\
A free filter on $(0,1)$ with the additional property
\begin{enumerate}
\item[(UF)] $S\in \Filter$ or $(0,1)\setminus S\in \Filter$, for each $S\subseteq (0,1)$
\end{enumerate}
is called a \emph{free ultrafilter} on $(0,1)$. By means of Zorn's lemma, one can show that every free filter can be extended to a free ultrafilter. If we replace $\Filter$ by a free ultrafilter and we identify two nets $(a_\eps)_\eps$, $(b_\eps)_\eps$ if they coincide on some $S\in\Filter$, then we obtain a model of nonstandard analysis (frequently, also free ultrafilters on other index sets than $(0,1)$ are used) \cite{Rob}.

The consequences of this technical change are very elegant: any formula in the formal language defined in section \ref{section_transfer} is then transferrable without restrictions, allowing $\vee$, $\neg$ and $\implies$ to be dealt with painlessly 
(usually, also the empty set is not excluded from the internal sets in this setting). E.g., by (UF),
\[[a_\eps]\ne [b_\eps]\iff \{\eps\in (0,1): a_\eps = b_\eps\}\notin \Filter \iff \{\eps\in (0,1): a_\eps\ne b_\eps\in \Filter\},\]
hence $a \ster\!\!\ne b$ is equivalent with $\neg (a=b)$.
By transfer, it follows also that $\ster\R$ is a totally ordered field, internal sets are closed under (finite) $\cup$ and $\setminus$, lemma \ref{lemma_finite} is immediate by the total order, \dots: summarizing, a lot of inconveniences disappear.\\
This refinement is particularly useful if one uses nonstandard analysis not so much as a model for singular `real world phenomena' (as in the nonlinear theory of generalized functions), but rather as a tool, an enrichment of language and objects, with the goal to prove results about the usual (=nongeneralized) objects in analysis in an easier way. E.g., in nonstandard analysis one obtains a very concise characterization of compactness in a (nongeneralized) topological space $X$:
\[K\subseteq X \text{ is compact }\iff (\forall x\in \ster K) (\exists y \in K) (x\approx y).\]
Concerning this use of nonstandard analysis, one may safely say that the model in \cite{SL} is deprecated, and the above text does not have any aspiration to compete with nonstandard analysis in that respect.

\end{document}